\documentclass[11pt]{amsart}
\usepackage{amssymb,amsmath,color}
\usepackage{tikz-dependency}
\usepackage{tikz}
\usepackage{enumerate}

\newtheorem{theorem}{Theorem}[section]
\newtheorem{proposition}[theorem]{Proposition}
\newtheorem{lemma}[theorem]{Lemma}
\newtheorem{definition}[theorem]{Definition}

\newtheorem{corollary}[theorem]{Corollary}
\newtheorem{example}[theorem]{Example}
\newtheorem{remark}[theorem]{Remark}

\numberwithin{equation}{section}

\newcommand{\SSSS}{{\mathfrak S}}

\newcommand{\Ppp}{{\mathbb P}}
\newcommand{\Rrr}{{\mathbb R}}
\newcommand{\Zzz}{{\mathbb Z}}

\DeclareMathOperator{\conv}{conv}
\DeclareMathOperator{\DP}{DP}
\DeclareMathOperator{\SP}{SP}
\DeclareMathOperator{\Head}{Head}
\DeclareMathOperator{\Tail}{Tail}
\DeclareMathOperator{\tw}{tw}

\begin{document}

\title{Simion's type $B$ associahedron is a pulling triangulation of the
  Legendre polytope}

\author{Richard EHRENBORG, G\'abor HETYEI and Margaret READDY}

\address{Department of Mathematics, University of Kentucky, Lexington,
  KY 40506-0027.\hfill\break \tt http://www.math.uky.edu/\~{}jrge/,
  richard.ehrenborg@uky.edu.}

\address{Department of Mathematics and Statistics,
  UNC-Charlotte, Charlotte NC 28223-0001.\hfill\break
\tt http://www.math.uncc.edu/\~{}ghetyei/,
ghetyei@uncc.edu.}

\address{Department of Mathematics, University of Kentucky, Lexington,
  KY 40506-0027.\hfill\break
\tt http://www.math.uky.edu/\~{}readdy/,
margaret.readdy@uky.edu.}


\date{\today}

\begin{abstract}
We show that Simion's type $B$ associahedron is combinatorially
equivalent to a pulling triangulation of a type $B$ root polytope
called the Legendre polytope. 
Furthermore, we show that every pulling triangulation of the
Legendre polytope yields a flag complex.
Our triangulation refines a decomposition
of the Legendre polytope given by Cho.  We extend  Cho's cyclic
group action to the triangulation in such a way that it corresponds to
rotating centrally symmetric triangulations of a regular $(2n+2)$-gon.
Finally, we present a bijection between the faces
of the Simion's type~$B$ associahedron
and Delannoy paths.
\end{abstract}

\maketitle

\section{Introduction}

Root polytopes arising as convex hulls of roots in a root system have
become the subject of intensive interest in recent
years~\cite{Ardila,Clark-Ehrenborg,
Gelfand-Graev-Postnikov,Hetyei-Legendre,Meszaros_I,Meszaros_II}.
Another important
area where geometry meets combinatorics is the study of noncrossing
partitions, associahedra and their generalizations. In this context
Simion~\cite{Simion} constructed a type $B$ associahedron whose facets
correspond to centrally symmetric triangulations of a regular $(2n+2)$-gon.
Burgiel and Reiner~\cite{Burgiel-Reiner}
described Simion's construction as
providing ``the first motivating example for an equivariant generalization
of fiber polytopes, that is, polytopal subdivisions which are
invariant under symmetry groups''.
It was recently observed by Cori and Hetyei~\cite{Cori-Hetyei} that the
face numbers in this type $B$ associahedron are the same as the face
numbers in any pulling triangulation of the boundary of a type~$B$ root
polytope, called the Legendre polytope in~\cite{Hetyei-Legendre}.

In this paper we show that the equality of these face numbers is not a
mere coincidence.
We prove the type~$B$ associahedron is combinatorially
equivalent to a pulling triangulation of the Legendre polytope $P_{n}$. The
convex hull of the positive roots among the vertices of the Legendre
polytope and of the origin is a type $A$ root polytope
$P_{n}^{+}$. Cho~\cite{Cho} has shown that the Legendre polytope $P_{n}$ may be
decomposed into copies of $P_{n}^{+}$ that meet only on their boundaries and
that there is a $\Zzz_{n+1}$-action on this decomposition. Our
triangulation representing the type~$B$ associahedron as a triangulation
of the Legendre polytope refines Cho's decomposition in such a way that
extends the $\Zzz_{n+1}$-action to the
triangulation. The effect of this $\Zzz_{n+1}$-action on the
centrally symmetric triangulations of the $(2n+2)$-gon is rotation.

Simion observed algebraically
that the number of $k$-dimensional
faces of
the type~$B$ associahedron is
given by the number of lattices paths
between $(0,0)$ and $(2n,0)$ taking 
$k$ up steps $(1,1)$,
$k$ down steps $(1,-1)$,
and
$n-k$ horizontal steps $(2,0)$.
Such paths are known as {\em balanced Delannoy paths}.
In Section~\ref{section_bijection}
we give a combinatorial proof
by providing
a bijection between the faces
of the type~$B$ associahedron
and Delannoy paths.
We give two presentations of this bijection,
one recursive and one non-recursive.

Our paper is structured as follows.
In the preliminaries we 
discuss
the Simion type~$B$ associahedron,
the Legendre polytope and pulling triangulations.
In Section~\ref{section_flag}
we show that every pulling triangulation of the Legendre polytope
is a flag complex.
We introduce
an arc representation of Simion's type~$B$ diagonals in
Section~\ref{section_arc} and obtain conditions for when a pair of
$B$-diagonals do not cross. A bijection
between the set of $B$-diagonals and the vertex set of the Legendre
polytope is obtained by the intermediary of our arc representation in
Section~\ref{section_embed}. We characterize when $B$-diagonals cross in
terms of crossing and nesting conditions on the arrows associated to the
vertices of the Legendre polytope. This characterization is used in
Section~\ref{section_triangulation} to define a triangulation of the boundary of
the Legendre polytope where each face corresponds to a face in the
type~$B$ associahedron. 
Since both complexes are flag and have
the same minimal non-faces,
we conclude that they are the same polytope.
We end this section by
describing the facets in our triangulation.
In Section~\ref{section_Cho}
we prove that our triangulation refines Cho's
decomposition and that his $\Zzz_{n+1}$-action corresponds to
rotating the regular $(2n+2)$-gon.   
In Section~\ref{section_bijection} we present a bijection
between faces of the
type~$B$ associahedron
and Delannoy paths.

We end the paper with comments and future research directions.

\section{Preliminaries}
\label{section_preliminaries}

\subsection{Simion's type $B$ associahedron}

Simion~\cite{Simion} introduced a simplicial complex denoted by~$\Gamma_{n}^{B}$ 
on  $n(n+1)$ vertices as follows. Consider a centrally 
symmetric convex $(2n+2)$-gon, and label its vertices in the clockwise
order with $1$, $2$, \ldots, $n$, $n+1$, $\overline{1}$, $\overline{2}$,
\ldots, $\overline{n}$,  $\overline{n+1}$. The vertices 
of $\Gamma_{n}^{B}$ are the {\em $B$-diagonals}, which are one of the two
following kinds: diagonals joining antipodal pairs of points, and
antipodal pairs of noncrossing diagonals. The diagonals joining
antipodal points are all pairs of the form
$\{i,\overline{i}\}$ satisfying $1\leq i\leq n+1$.
Simion calls such a $B$-diagonal a {\em diameter}.
The $B$-diagonals
that are antipodal pairs of noncrossing  diagonals are either of the form
$\{\{i,j\},\{\overline{i},\overline{j}\}\}$ satisfying $1\leq
i<i+1<j\leq n+1$ or of the form $\{\{i,\overline{j}\}, 
\{\overline{i},j\}\}$ satisfying $1\leq j< i\leq n+1$.

The simplicial complex  $\Gamma_{n}^{B}$ is the family of sets of pairwise
noncrossing $B$-diagonals. Simion showed the
simplicial complex $\Gamma_{n}^{B}$ is the boundary complex of an
$n$-dimensional convex polytope. 
This polytope
is also known as
the Bott--Taubes polytope~\cite{Bott_Taubes}
and
the cyclohedron~\cite{Markl}.
Simion also computed the face
numbers and $h$-vector.
See Theorem~1, Proposition~1
and Corollary~1 in~\cite{Simion}, respectively. These numbers turn
out to be identical with the face numbers and $h$-vector of any pulling
triangulation of the Legendre polytope.
We will discuss this polytope in the next subsection. We end with a fact
that is implicit in the work of  
Simion~\cite[Section~3.3]{Simion}.
\begin{lemma}[Simion]
\label{lemma_diameter}
Each facet $\Gamma_{n}^{B}$
of Simion's type~$B$ associahedron
contains exactly one $B$-diagonal of the form
$\{i,\overline{i}\}$ connecting an antipodal pair of points.  
\end{lemma}

\subsection{The Legendre polytope or ``full'' type $A$ root polytope} 

Consider an $(n+1)$-dimensional Euclidean space
with orthonormal basis $\{e_1, e_2, \ldots, e_{n+1}\}$.
The convex hull of the vertices
$\pm 2e_1$, \ldots, $\pm 2 e_{n+1}$
is an $(n+1)$-dimensional cross-polytope.
The intersection of this cross-polytope with the hyperplane
$x_{1} + x_{2} + \cdots + x_{n+1} = 0$
is an $n$-dimensional centrally
symmetric polytope $P_{n}$ first studied by Cho~\cite{Cho}.
It is called
the {\em Legendre polytope}
in the work of Hetyei~\cite{Hetyei-Legendre},
since the polynomial
$\sum_{j=0}^{n} f_{j-1} \cdot ((x-1)/2)^{j}$
is the $n$th Legendre polynomial, where $f_{i}$ is the number of
$i$-dimensional faces in any pulling triangulation of the boundary of $P_{n}$.
See Lemma~\ref{lemma_any_pull} below.
Furthermore, it is called
the ``full'' type $A$ root polytope in the work of
Ardila--Beck--Ho\c{s}ten--Pfeifle--Seashore~\cite{Ardila}. It has
$n(n+1)$ vertices consisting of all points of the 
form $e_{i}-e_{j}$ where $i\neq j$. 

We use the shorthand notation $(i,j)$
for the vertex $e_{j}-e_{i}$ of the Legendre polytope $P_{n}$.
We may think of these vertices as the set of all
directed nonloop edges on the vertex set $\{1,2,\ldots,n+1\}$.
A subset of
these edges is contained in some face of $P_{n}$
exactly when there is no
$i \in \{1,2,\ldots,n+1\}$ that is both the head and the tail of a
directed edge.
Equivalently, the faces are described as follows.
\begin{lemma}
The faces of the Legendre polytope $P_{n}$
are of the form
$\conv(I \times J) = \conv(\{(i,j) \: : \: i \in I, j \in J\})$
where 
$I$ and $J$ are two non-empty disjoint subsets
of the set $\{1,2,\ldots,n+1\}$.
The dimension of a face is given by
$|I| + |J| - 2$.
A face is a facet if and only if the union of $I$ and $J$ is
the set $\{1,2,\ldots,n+1\}$.
\label{lemma_facets}
\end{lemma}
\noindent
Especially, when the two sets $I$ and $J$
both have cardinality two, the associated 
face is geometrically a square.
Furthermore, the other two-dimensional faces
are equilateral triangles.

Affine independent subsets of vertices of faces of the Legendre
polytope are
easy to describe. A set $S=\{(i_1,j_1),(i_2,j_2),\ldots, (i_{k},j_{k})\}$ is a
$(k-1)$-dimensional simplex if and only if, 
disregarding the orientation of the directed
edges, the set $S$ contains no
cycle, that is, it is a forest~\cite[Lemma~2.4]{Hetyei-Legendre}. 

The Legendre polytope $P_{n}$ contains the polytope $P_{n}^{+}$, defined as the convex
hull of the origin and the set of points $e_{i}-e_{j}$, where $i<j$. The
polytope $P_{n}^{+}$ was first studied by Gelfand, Graev and
Postnikov~\cite{Gelfand-Graev-Postnikov} and later by
Postnikov~\cite{Postnikov-P}.
Some of the results on $P_{n}^{+}$ may be easily
generalized to~$P_{n}$.

\subsection{Pulling triangulations}

The notion of pulling triangulations is 
originally due to Hudson~\cite[Lemma~1.4]{Hudson}.
For more modern formulations, see~\cite[Lemma~1.1]{Stanley-Dec}
and~\cite[End of Section~2]{Athanasiadis}.
We refer to~\cite[Section~2.3]{Hetyei-Legendre} for the version
presented here.

For a polytopal complex ${\mathcal P}$ 
and a vertex $v$ of ${\mathcal P}$,
let ${\mathcal P}- v$ be the complex consisting
of all faces of~${\mathcal P}$ not
containing the vertex $v$.
Also for a facet $F$ let
${\mathcal P}(F)$ be the complex of all
faces of ${\mathcal P}$ contained in~$F$.

\begin{definition}[Hudson]
\label{definition_pull}
Let  ${\mathcal P}$  be a polytopal complex and let $<$ be a linear
order on the set $V$ of its vertices. The pulling triangulation
$\triangle({\mathcal P})$ with respect to $<$ is defined
recursively as follows. We set 
$\triangle({\mathcal P})={\mathcal P}$ if ${\mathcal P}$ consists
of a single vertex. Otherwise let $v$ be the least element of $V$ with
respect to $<$ and set 
$$
\triangle({\mathcal P})
=
\triangle({\mathcal P}- v)\cup
\bigcup_F \left\{ \conv(\{v\}\cup G)\::\: G\in
\triangle({\mathcal P}(F))\right\} ,
$$
where the union runs over the facets $F$ not containing $v$ of the
maximal faces of ${\mathcal P}$ which contain~$v$.
The triangulations
$\triangle({\mathcal P}- v)$ and $\triangle({\mathcal P}(F))$ are
with respect to the order~$<$
restricted to their respective vertex sets.
\end{definition}

\begin{theorem}[Hudson]
The pulling triangulation $\triangle({\mathcal P})$
is a triangulation of
the polytopal complex~${\mathcal P}$
without introducing any new vertices.
\label{theorem_pull}
\end{theorem}

In particular, any pulling triangulation of the
boundary of $P_{n}$  is compressed as defined by
Stanley~\cite{Stanley-Dec}, and has the same face 
numbers~\cite[Corollary 4.11]{Hetyei-Legendre}. This important fact
and the analogous statement for $P_{n}^{+}$ is a direct consequence of the
following two fundamental results~\cite{Heller,Heller-Hoffman,Stanley-Dec}.
\begin{proposition}[Stanley]
\label{proposition_compress}
Suppose that one of the vertices of a polytope $P$ is the origin and that the
matrix whose rows are the coordinates of   
the vertices of $P$ is totally unimodular.
Let $<$ be any ordering on the vertex set of~$P$ such that the
origin is the least vertex with respect to $<$. Then the pulling order
$<$ is compressed, that is, all of the facets in the induced triangulation
have the same relative volume.
\end{proposition} 
\begin{theorem}[Heller]
\label{theorem_Heller}
The incidence matrix of a directed graph is totally unimodular.
\end{theorem}

\subsection{Face vectors of pulling triangulations of
the Legendre polytope}

Among all triangulations of the boundary of the Legendre polytope $P_{n}$
obtained by pulling
the vertices, counting faces is most easily performed for the
{\em lexicographic triangulation}
in which we pull $(i,j)$ before $(i',j')$ exactly when
$i<i'$ or when $i=i'$ and $j<j'$. Counting faces in this
triangulation amounts to counting lattice paths;
see~\cite[Lemma~5.1]{Hetyei-Legendre}
and~\cite[Proposition~17]{Ardila}. From this we
obtain the following expression for the face 
numbers~\cite[Theorem~5.2]{Hetyei-Legendre}. 
\begin{lemma}[Hetyei]
For any pulling triangulation of the boundary of $P_{n}$,
the number $f_{j-1}$ of $(j-1)$-dimensional faces is 
\begin{equation}
\label{equation_f}
f_{j-1}=\binom{n+j}{j}\binom{n}{j} \quad \mbox{for $0\leq j\leq n$}.
\end{equation}    
\label{lemma_any_pull}  
\end{lemma}
It was first noted in~\cite{Cori-Hetyei} that
these face numbers are the
same as that of Simion's type $B$ associahedron~$\Gamma_{n}^{B}$.
Theorem~\ref{theorem_r_pull} will explicitly explain this fact 
by showing that
$\Gamma_{n}^{B}$ can be realized
as a pulling triangulation of the Legendre polytope.

\section{The flag property}
\label{section_flag}

Recall that a simplicial complex is a {\em flag complex} if every
minimal nonface has two elements.
The main result of this section is the following.
\begin{theorem}
\label{theorem_pull_flag}  
Every pulling triangulation of the boundary of the Legendre polytope
$P_{n}$ is a flag simplicial complex.
\end{theorem}

To prove this theorem and Theorem~\ref{theorem_r_pull}, we need
the following observation. 
\begin{lemma}
Let $\{x_{1},x_{2},y_{1},y_{2}\}$ be a four element subset of the set
$\{1,2, \ldots, n+1\}$. Then the set
$\{x_{1},x_{2}\} \times \{y_{1},y_{2}\} = 
\{(x_{1},y_{1}),(x_{1},y_{2}),(x_{2},y_{2}),(x_{2},y_{1})\}$
is the vertex set of a square face of the Legendre polytope $P_{n}$, and the sets
$\{(x_{1},y_{1}), (x_{2},y_{2})\}$ and $\{(x_{1},y_{2}), (x_{2},y_{1})\}$
are the diagonals of this square. For any pulling triangulation the
diagonal containing the vertex that was pulled first is an edge of the
triangulation and the other diagonal is not an edge.
\label{lemma_square}
\end{lemma}
Theorem~\ref{theorem_pull_flag} may be
rephrased as follows.
\begin{theorem}
Let $<$ be any order on the vertices of the Legendre polytope $P_{n}$
and consider
the pulling triangulation of the boundary of $P_{n}$
induced by this order.
Suppose we are given a set of vertices
$\{(u_{1},v_{1}),(u_{2},v_{2}),\ldots,(u_{k},v_{k})\}$
such that any pair
$\{(u_{i},v_{i}),(u_{j},v_{j})\}$ is an edge in the pulling triangulation.
Then this set is a face of the pulling triangulation.
\end{theorem}  
\begin{proof}
Assume that the pulling order on the set of vertices is given by
$(u_{1},v_{1}) < (u_{2},v_{2}) < \cdots < (u_{k},v_{k})$.
We prove the statement by induction on $k$. The statement is
directly
true for $k\leq 2$. Assume from now on $k\geq 3$.
Observe first that
the sets $\{u_{1},\ldots,u_{k}\}$ and $\{v_{1},\ldots,v_{k}\}$ must be
disjoint. Indeed, if for any $i \neq j$ we have $u_{i}=v_{j}$ then the pair of
vertices $\{(u_{i},v_{i}), (u_{j},v_{j})\}$ is not an edge, contradicting
our assumption. Thus our set of vertices is contained in the face
$\conv(\{u_{1},\ldots,u_{k}\}\times \{v_{1},\ldots,v_{k}\})$
of the polytope~$P_{n}$.
Note that the lists $u_{1},\ldots,u_{k}$ and $v_{1},\ldots,v_{k}$ may
contain repeated elements.

We next show that $(u_{1},v_{1})$ is the least element with respect of the
order $<$ in the set $\{u_{1},\ldots,u_{k}\}\times \{v_{1},\ldots,v_{k}\}$. To
prove this, suppose $(u_{i},v_{j})$ is the least element. If $i=j$ then
either $i=1$ or we may use that we were given $(u_{1},v_{1})<(u_{i},v_{i})$ for
all $i>1$. If $i\neq j$ but $u_{i}=u_{j}$ then we have $(u_{i}, v_{j})=(u_{j},
v_{j})$ and again we are done since $(u_{j},v_{j})>(u_{1},v_{1})$ if $j>1$.
Similarly, if $i\neq j$ but $v_{i}=v_{j}$ then we may use $(u_{i},
v_{j})=(u_{i},v_{i})$. Finally if $i\neq j$, $u_{i}\neq u_{j}$ and $v_{i}\neq v_{j}$
hold, then 
apply Lemma~\ref{lemma_square} to
the square with vertex set
$(u_{i},v_{i})$, $(u_{i},v_{j})$, $(u_{j},v_{j})$ and $(u_{j},v_{i})$. 
Since $(u_{i},v_{j})$ is the least 
with respect to the pulling order,
the diagonal $\{(u_{i},v_{i}), (u_{j},v_{j})\}$ is not an edge of the
triangulation, contradicting our assumption.
Hence we conclude $(u_{1},v_{1})$ is the least vertex in the set
$\{u_{1},\ldots,u_{k}\} \times \{v_{1},\ldots,v_{k}\}$.

We claim the smallest face of the
boundary of $P_{n}$ containing
the vertex set $\{(u_{2},v_{2}), \ldots, (u_{k}, v_{k})\}$
does not contain the vertex $(u_{1},v_{1})$.
Assume, by way of contradiction, that
$\{u_{1}, u_{2}, \ldots,u_{k}\}=\{u_{2},\ldots,u_{k}\}$ and
$\{v_{1}, v_{2}, \ldots,v_{k}\}=\{v_{2},\ldots,v_{k}\}$ hold.
Then there is an index $i>1$ and
an index $j>1$ such that $u_{1}=u_{i}$ and $v_{1}=v_{j}$ hold.
Note that we must
also have $v_{1} \neq v_{i}$ and $u_{1}\neq u_{j}$.  
Consider the square face with the following four
vertices: $(u_{1},v_{1})$, $(u_{1},v_{i})=(u_{i},v_{i})$,  $(u_{j},v_{i})$ and
$(u_{j},v_{j})=(u_{j},v_{1})$. The first vertex that was pulled is
$(u_{1},v_{1})$. By Lemma~\ref{lemma_square}
the edge $\{(u_{i},v_{i}), (u_{j},v_{j})\}$
is not an edge in the pulling triangulation,
contradicting our assumptions.  

By the induction hypothesis
the set $\{(u_{2},v_{2}), \ldots, (u_{k},v_{k})\}$
is a face in the polytopal complex
$\partial P_{n} - \{w \: : \: w \leq (u_{1},v_{1})\}$.
By definition, we obtain that
$\{(u_{1},v_{1})\} \cup \{(u_{2},v_{2}), \ldots, (u_{k},v_{k})\}$
is a face of the pulling triangulation.
\end{proof}  

Since the Cartesian product of an $m$-dimensional simplex
and an $n$-dimensional simplex is a face of
the Legendre polytope of dimension $m+n+1$, we obtain the following
corollary. 
\begin{corollary}
Every pulling triangulation of the Cartesian product of
two simplices is a flag complex.
\end{corollary}

\section{The arc representation of $\Gamma_{n}^{B}$}
\label{section_arc}

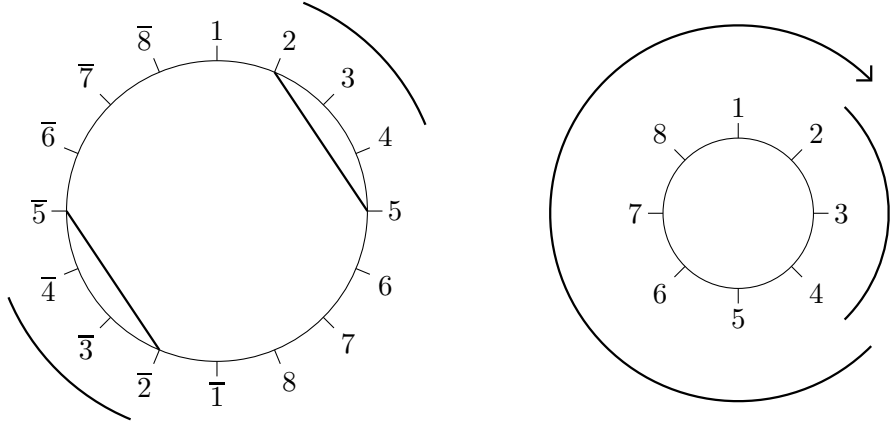
\begin{figure}
\begin{tikzpicture}
\draw(0mm,0mm) circle (20mm);

\draw(0mm,20mm) -- (0mm,22mm);
\draw(7.653mm,18.477mm) -- (8.419mm,20.325mm);
\draw(14.142mm,14.142mm) -- (15.556mm,15.556mm);
\draw(18.477mm,7.653mm) -- (20.325mm,8.419mm);
\draw(20mm,0mm) -- (22mm,0mm);
\draw(18.477mm,-7.653mm) -- (20.325mm,-8.419mm);
\draw(14.142mm,-14.142mm) -- (15.556mm,-15.556mm);
\draw(7.653mm,-18.477mm) -- (8.419mm,-20.325mm);
\draw(0mm,-20mm) -- (0mm,-22mm);
\draw(-7.653mm,-18.477mm) -- (-8.419mm,-20.325mm);
\draw(-14.142mm,-14.142mm) -- (-15.556mm,-15.556mm);
\draw(-18.477mm,-7.653mm) -- (-20.325mm,-8.419mm);
\draw(-20mm,0mm) -- (-22mm,0mm);
\draw(-18.477mm,7.653mm) -- (-20.325mm,8.419mm);
\draw(-14.142mm,14.142mm) -- (-15.556mm,15.556mm);
\draw(-7.653mm,18.477mm) -- (-8.419mm,20.325mm);

\label(0mm,0mm)
[label distance=100mm] 
\node[circle,
label=90:$1$,
label=70:$2$,
label=45:$3$,
label=20:$4$,
label=0:$5$, 
label=340:$6$, 
label=315:$7$, 
label=290:$8$, 
label=270:$\overline{1}$, 
label=250:$\overline{2}$, 
label=225:$\overline{3}$, 
label=200:$\overline{4}$, 
label=180:$\overline{5}$, 
label=160:$\overline{6}$,
label=135:$\overline{7}$,
label=110:$\overline{8}$
]{\hspace*{40mm}}; 

\draw[line width=0.3mm](7.653mm,18.477mm) -- (20mm,0mm);
\draw[line width=0.3mm](-7.653mm,-18.477mm) -- (-20mm,0mm);

\draw[line width=0.3mm] (27.716mm,11.480mm)
arc (22.5:67.5:30mm);
\draw[line width=0.3mm] (-27.716mm,-11.480mm)
arc (202.5:247.5:30mm);

\end{tikzpicture}
\hspace*{10mm}
\raisebox{0mm}{
\begin{tikzpicture}

\draw(0mm,0mm) circle (10mm);

\draw(0mm,10mm) -- (0mm,12mm);
\draw(7.071mm,7.071mm) -- (8.485mm,8.485mm);
\draw(10mm,0mm) -- (12mm,0mm);
\draw(7.071mm,-7.071mm) -- (8.485mm,-8.485mm);
\draw(0mm,-10mm) -- (0mm,-12mm);
\draw(-7.071mm,-7.071mm) -- (-8.485mm,-8.485mm);
\draw(-10mm,0mm) -- (-12mm,0mm);
\draw(-7.071mm,7.071mm) -- (-8.485mm,8.485mm);

\label(0mm,0mm)
[label distance=100mm] 
\node[circle,
label=90:$1$,
label=45:$2$,
label=0:$3$, 
label=315:$4$, 
label=270:$5$, 
label=225:$6$, 
label=180:$7$, 
label=135:$8$,
]{\hspace*{20mm}}; 

\draw[line width=0.3mm] (14.142mm,-14.142mm)
arc (315:405:20mm);

\draw[line width=0.3mm] (17.677mm,-17.677mm)
arc (315:45:25mm);

\draw[line width=0.3mm] (15.677mm,17.677mm)
-- (17.677mm,17.677mm) -- (17.677mm,19.677mm);

\end{tikzpicture}
}
\caption{The arc representation of the $B$-diagonal
consisting of the two diagonals $\{2,5\}$ and
$\{\overline{2},\overline{5}\} = \{10,13\}$
is the two arcs $[2,4]$ and $[\overline{2},\overline{4}] = [10,12]$.
By considering the arcs modulo $n+1=8$
(see the second circle)
we obtain that this $B$-diagonal is represented by the arrow~$(4,2)$.}
\label{figure_one}
\end{figure}

In this section we describe a representation of $\Gamma_{n}^{B}$ as a
simplicial complex whose vertices
are centrally symmetric pairs of ``arcs''
on a circle. This representation has a natural circular symmetry.

Consider a regular $(2n+2)$-gon whose vertices are labeled $1,2,
\ldots,n+1,\overline{1},\overline{2}, \ldots,
\overline{n+1}$ in the clockwise order. Identify each vertex 
$\overline{i}$ with $n+1+i$ for $i=1,2,\ldots, n+1$. Subject to this
identification, each $B$-diagonal, that is, a pair of diagonals,
may be represented as an unordered pair of diagonals of the form
$\{\{u,v\},\{u+n+1,v+n+1\}\}$ for some  $\{u,v\}\subseteq\{1,2,\ldots,
2n+2\}$, where addition is modulo~$2n+2$. For $B$-diagonals
$\{k, \overline{k}\}$ joining antipodal points, the unordered pair
$\{\{k,k+n+1\}, \{k+n+1,k+2n+2\}\}$ contains two copies of the same
two-element set.

For any two points $x$ and $y$ on the circle $\Rrr/(2n+2)\Zzz$ 
which are not antipodal, let
$[x,y]$ denote the shortest arc from $x$ to $y$.

\begin{definition}
\label{definition_arc-representation}
We define the {\em arc-representation} on the vertices of
$\Gamma_{n}^{B}$ as
follows. Subject to the above identifications, represent the 
$B$-diagonal $\{\{u,v\}, \{u+n+1, v+n+1\}\}$ with the centrally
symmetric pair of arcs $\{[u,v-1], [u+n+1,v+n]\}$ on the 
circle $\Rrr/(2n+2)\Zzz$.
\end{definition}
The representation above is well-defined: for any pair of
vertices $u$ and $v$,
the arc $[u,v-1]$ and the arc $[u+n+1,v+n]$ form a centrally symmetric
pair, so the definition is independent of the selection of the element
of $\{\{u,v\}, \{u+n+1,v+n+1\}\}$. Note that for $B$-diagonals of
the form $\{\{k,n+1+k\},\{k,n+1+k\}\}$ corresponding to antipodal pairs
of points, the union of the arcs $[k,k+n]$ and $[k+n+1,k-1]$ is not
the full circle.
 
See Figure~\ref{figure_one} for an example where
$n=7$ with the $B$-diagonal $\{\{2,5\},\{\overline{2},\overline{5}\}\}$.
 
\begin{lemma}
The arc-representation of the vertices of $\Gamma_{n}^{B}$ is one-to-one:
distinct $B$-diagonals are mapped to distinct centrally symmetric pairs 
of arcs.
\end{lemma}  
Indeed, the pair $\{[u,v-1],[u+n+1,v+n]\}$ can only be the image of the
unordered pair $\{\{u,v\},\{u+n+1,v+n+1\}\}$. 

The following theorem plays an important role in 
connecting the type $B$ associahedron with the Legendre polytope.
\begin{theorem}
\label{theorem_across}
The $B$-diagonal represented by the pair of arcs 
$\{[u_{1}, v_{1}-1], [u_{1}+n+1,v_{1}+n]\}$ and the $B$-diagonal
represented by the pair of arcs $\{[u_{2}, v_{2}-1], [u_{2}+n+1,v_{2}+n]\}$ are
noncrossing if and only if for either arc $I\in \{[u_{1},
  v_{1}-1],[u_{1}+n+1,v_{1}+n]\}$ and for either arc $J\in \{[u_{2}, v_{2}-1],
[u_{2}+n+1,v_{2}+n]\}$, the arcs $I$ and $J$ are either nested or disjoint.  
\end{theorem}
\begin{proof}
Assume first that the two $B$-diagonals cross. Since we may replace
$\{u_{i},v_{i}\}$ with $\{u_{i}+n+1,v_{i}+n+1\}$ if necessary, without loss of
generality we may assume that the diagonal $\{u_{1},v_{1}\}$ crosses the
diagonal $\{u_{2},v_{2}\}$.
Exactly one of the
endpoints of the diagonal $\{u_{2},v_{2}\}$ must then belong to the arc
$[u_{1}+1,v_{1}-1]$ and the other one does not belong even to the larger arc
$[u_{1},v_{1}]$. If $u_{2}\in [u_{1}+1,v_{1}-1]$ and $v_{2}\not\in [u_{1},v_{1}]$ 
then we have
$$
[u_{1}, v_{1}-1]\cap [u_{2}, v_{2}-1]=[u_{2},v_{1}-1].
$$
This arc contains $u_{2}$, but does not contain $u_{1}$ (since $u_{2}\in
[u_{1}+1,v_{1}-1]$), nor does it contain $v_{2}-1\not \in [u_{1},v_{1}-1]$. The
arcs $[u_{1}, v_{1}-1]$  and $[u_{2}, v_{2}-1]$ do not contain each other
and they are not disjoint.

If $v_{2}\in [u_{1}+1,v_{1}-1]$ and $u_{2}\not\in [u_{1},v_{1}]$ 
then we have
$$
[u_{1}, v_{1}-1]\cap [u_{2}, v_{2}-1]=[u_{1},v_{2}-1].
$$
Similar to the previous case, we obtain that the arcs $[u_{1}, v_{1}-1]$
and $[u_{2}, v_{2}-1]$ do not contain each other and they are not disjoint.

Assume next that the two diagonals do not cross. If $u_{2}\in
[u_{1}+1,v_{1}-1]$ then we must have $v_{2}\in [u_{2}+1,v_{1}]$, implying 
$$
[u_{1}, v_{1}-1]\cap [u_{2}, v_{2}-1]=[u_{2},v_{2}-1].
$$
Similarly, if $v_{2}\in [u_{1}+1,v_{1}-1]$ then we must have $u_{2}\in
[u_{1},v_{2}]$, implying
$$
[u_{1}, v_{1}-1]\cap [u_{2}, v_{2}-1]=[u_{2},v_{2}-1].
$$
Finally, if neither $u_{2}$ nor $v_{2}$ belongs to $[u_{1}+1,v_{1}-1]$ then
either the arc $[u_{2},v_{2}-1]$ contains the arc $[u_{1},v_{1}-1]$ or it is disjoint from it. 
The above argument remains valid if we replace either (or both) of
$[u_{1}, v_{1}-1]$ and $[u_{2}, v_{2}-1]$ with the arc $[u_{1}+n+1, v_{1}+n]$ or
$[u_{2}+n+1, v_{2}+n]$, respectively. 
\end{proof}
\begin{corollary}
\label{corollary_across}  
The $B$-diagonal represented by the pair of arcs 
$\{[u_{1}, v_{1}-1], [u_{1}+n+1,v_{1}+n]\}$ and the $B$-diagonal
represented by the pair of arcs $\{[u_{2}, v_{2}-1], [u_{2}+n+1,v_{2}+n]\}$ are
noncrossing if and only the set $[u_{1}, v_{1}-1]\cup [u_{1}+n+1,v_{1}+n]$ and
the set $[u_{2}, v_{2}-1]\cup [u_{2}+n+1,v_{2}+n]$ are nested or disjoint. 
\end{corollary}

\section{Embedding $\Gamma_{n}^{B}$ as a family of simplices on $\partial P_{n}$}
\label{section_embed}

In this section we describe a way to represent the boundary complex
$\Gamma_{n}^{B}$ of Simion's type $B$ associahedron as a family
of simplices on the vertex set of the Legendre polytope $P_{n}$. We do
this so that each simplex is contained in a face of the
boundary $\partial P_{n}$ of $P_{n}$. In Section~\ref{section_triangulation} we will
prove that our map represents the boundary complex of the type $B$
associahedron as a pulling triangulation of $\partial P_{n}$.
 
We begin by defining a bijection between the vertex set of $\Gamma_{n}^{B}$
and that of $P_{n}$.
Recall that we use
the shorthand notation $(i,j)$ for the vertex $e_{j}-e_{i}$ of $P_{n}$.
We refer to $(i,j)$ as {\em the arrow  from $i$ to $j$}. Using the term
``arrow'' as opposed to ``directed edge'' will eliminate the
confusion that $e_{j}-e_{i}$ is a vertex of $P_{n}$.
Instead we think of it as an arrow from $i$ to $j$ in the complete
directed graph on the vertex set $\{1,2,\ldots,n+1\}$ having no loops.

\begin{definition}
\label{definition_arrow_representation}
Let $\{[i,j], [\overline{i},\overline{j}]\}$ be the arc representation
of a $B$-diagonal in $\Gamma_{n}^{B}$,
where $1 \leq i \leq n+1$ and $i < j$.
Define {\em the arrow representation} 
of this $B$-diagonal in $P_{n}$ to be the arrow~$(j,i)$.
\end{definition}
In other words,
the arrow encodes the complement of
the image of the arcs in the circle $\Rrr/(n+1)\Zzz$.
We refer to
the second circle in Figure~\ref{figure_one}
for the continuation of the example 
of the $B$-diagonal $\{\{2,5\},\{\overline{2},\overline{5}\}\}$.

When making this definition explicit for
a $B$-diagonal $\{\{i,j\},\{\overline{i}, \overline{j}\}\}$,
we obtain several cases:
\begin{enumerate}[(1)]
\item
For each $i$ satisfying $2 \leq i \leq n+1$,
represent the $B$-diagonal
$\{i,\overline{i}\}$ connecting two antipodal points with
the arrow~$(i-1,i)$. Represent the $B$-diagonal $\{1,\overline{1}\}$
with the arrow~$(n+1,1)$.
\item
For each $i$ and $j$ satisfying $1\leq i<i+1<j\leq n+1$,  
represent the $B$-diagonal $\{\{i,j\},\{\overline{i}, \overline{j}\}\}$
with the arrow~$(j-1,i)$.
\item
For each $i$ and $j$ satisfying $2\leq j<i\leq n+1$, 
represent the $B$-diagonal
$\{\{i,\overline{j}\}, \{\overline{i},j\}\}$
with the arrow~$(j-1,i)$. For each $i$ satisfying $2 \leq i \leq n+1$,
represent the $B$-diagonal $\{\{1,\overline{i}\},\{\overline{1},i\}\}$
with the arrow~$(n+1,i)$.
\end{enumerate}

This representation yields a bijection between 
$B$-diagonals and arrows. The inverse map is given as follows:
\begin{enumerate}[(a)]
\item
For $i$ satisfying $1 \leq i \leq n$,
the arrow $(i,i+1)$ represents the $B$-diagonal $\{i+1,\overline{i+1}\}$
and
the arrow $(n+1,1)$ represents the $B$-diagonal $\{1,\overline{1}\}$.
\item
For each $i$ and $j$ satisfying $1\leq i<j\leq n$, the arrow
$(j,i)$ represents the $B$-diagonal
$\{\{i,j+1\},\{\overline{i},\overline{j+1}\}\}$. 
\item
For each $i$ and $j$ satisfying $1\leq j<i-1<i\leq n+1$, the
arrow $(j,i)$ represents the $B$-diagonal
$\{\{i,\overline{j+1}\},\{\overline{i},j+1\}\}$, and for  
$2 \leq i \leq n+1$, the arrow $(n+1,i)$
represents the $B$-diagonal $\{\{1,\overline{i}\}, \{\overline{1},i\}\}$.   
\end{enumerate}

The $B$-diagonals in
item~(1) of
Definition~\ref{definition_arrow_representation}
may be thought of as a ``degenerate case'' of the $B$-diagonals
in item~(3). In the case when $i=j$, the set
$\{\{\overline{i},j\},\{i,\overline{j}\}\}$ becomes a singleton, and the
rules in item~(1) may be obtained by extending the rules
in item~(3) in an obvious way. On the other hand,
Definition~\ref{definition_arrow_representation} has a simpler form in
terms of the arc-representation 
described in Section~\ref{section_arc} and in terms of the following
continuous map.

\begin{definition}
\label{definition_pi}
Define the map
$\pi: \Rrr/(2n+2)\Zzz \longrightarrow \Rrr/(n+1)\Zzz$
to be the modulo $n+1$ map.
Furthermore, identify the circle $\Rrr/(n+1)\Zzz$
with the half-open interval
$(0,n+1]$.
Thus $\pi$ sends each $x \in (0,n+1]$ to~$x$
and each $x \in (n+1,2n+2]$ to $x-n-1$.  
\end{definition}
Observe that the map $\pi$ depends on $n$.
However, we suppress this dependency
by not writing $\pi_{n}$. 
Also observe that the map $\pi$ is a two-to-one mapping:
for each $y\in \Rrr/(n+1)\Zzz$ we have $|\pi^{-1}(y)|=2$.

\begin{remark}
\label{remark_arrows}
{\em  For any pair of arcs $\{[u,v-1],[u+n+1,v+n]\}$ there is a unique way to
select~$u$ to be an element of the set $\{1,2,\ldots,n+1\}$,
that is, $\pi(u)=u$. We may distinguish two
cases depending upon whether the arc $[u,v-1]$ is a subset of the arc
$[u,n+1]$ or not.
\begin{enumerate}[(i)]
\item
If $[u,v-1]\subseteq [u,n+1]$ then visualize
the set $\pi([u,v-1])=\pi([u+n+1,v+n])$ as the subinterval
$[u,v-1]$ of $(0,n+1]$. The direction of both arcs
$[u,v-1]$ and $[u+n+1,v+n]$ corresponds to parsing the interval
$[u,v-1]$ in increasing order. Adding the associated backward arrow
$(v-1,u)$ closes a directed cycle with this directed interval.
$$
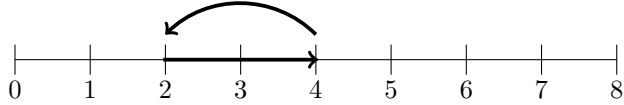
\begin{tikzpicture}
\draw (0,0) -- (8,0);
\foreach \x in {0,...,8}
\draw (\x,-0.2) -- (\x,0.2);
\foreach \x in {0,...,8}
\node[black] at (\x,-0.4) {\small$\x$};

\node[anchor=east] at (4.3,0) (text) {};
\node[anchor=west] at (1.7,0) (description) {};
\draw[line width=0.5mm] (description) edge[out=0,in=180,->] (text);

\node[anchor=east] at (2,0.2) (text) {};
\node[anchor=west] at (4,0.2) (description) {};
\draw[line width=0.5mm] (description) edge[out=135,in=45,->] (text);
\end{tikzpicture}
$$
As an example, when $n=7$ then the $B$-diagonal 
$\{\{2,5\},\{\overline{2},\overline{5}\}\}$
is represented by the backward arrow $(4,2)$
as drawn above on the interval $(0,8]$.
\item
If the
arc $[u,v-1]$ is not contained in the arc 
$[u,n+1]$ then $n+1 < v-1 < u+n+1$.
The integer $\pi(v-1)=v-1-(n+1)$ is congruent to $v-1$ modulo~$(n+1)$
and satisfies $1\leq \pi(v-1)<u$.
The image of the arc $[u,v-1]$ under $\pi$,
that is,
$\pi([u,v-1])=\pi([u+n+1,v+n])$ is then the subset 
$(0,\pi(v)-1]\cup [u,n+1]$ of the interval $(0,n+1]$.
We may consider $(0,\pi(v)-1]\cup [u,n+1]$ as a ``wraparound
interval'' modulo~$n+1$ from $u$ to $\pi(v-1)$.
The direction of both pieces
corresponds to listing the elements of this ``wraparound
interval'' in increasing order modulo $n+1$.  
Adding the associated forward arrow $(\pi(v-1),u)$ closes a directed
cycle with the directed wraparound interval.
$$
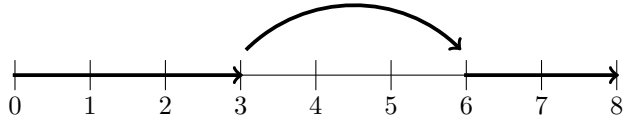
\begin{tikzpicture}
\draw (0,0) -- (8,0);
\foreach \x in {0,...,8}
\draw (\x,-0.2) -- (\x,0.2);
\foreach \x in {0,...,8}
\node[black] at (\x,-0.4) {\small$\x$};

\node[anchor=east] at (3.3,0) (text) {};
\node[anchor=west] at (-0.3,0) (description) {};
\draw[line width=0.5mm] (description) edge[out=0,in=180,->] (text);

\node[anchor=east] at (8.3,0) (text) {};
\node[anchor=west] at (5.7,0) (description) {};
\draw[line width=0.5mm] (description) edge[out=0,in=180,->] (text);

\node[anchor=east] at (6.2,0.2) (text) {};
\node[anchor=west] at (2.8,0.2) (description) {};
\draw[line width=0.5mm] (description) edge[out=45,in=135,->] (text);
\end{tikzpicture}
$$
For instance, when $n = 7$
the $B$-diagonal $\{\{4,\overline{6}\},\{\overline{4},6\}\}$
yields the forward arrow $(3,6)$.
\end{enumerate}
}
\end{remark}

\begin{proposition}
\label{proposition_crossing_pi}
The $B$-diagonal represented by the arrow $(\pi(v_{1}-1),\pi(u_{1}))$ and
the $B$-diagonal represented by the arrow $(\pi(v_{2}-1),\pi(u_{2}))$
are noncrossing if and only if the images $\pi([u_{1},v_{1}-1])$ and
$\pi([u_{2},v_{2}-1])$ are disjoint or contain each other.  
\end{proposition}
\begin{proof}
The statement is an easy consequence of the following observation.
Both arcs of the arc representation $\{[u,v-1],[u+n+1,v+n]\}$ of a
$B$-diagonal are mapped onto the same arc $[\pi(u),\pi(v-1)]$ by
$\pi$ and $[u,v-1]\cup [u+n+1,v+n]=\pi^{-1}([\pi(u),\pi(v-1)])$. 
\end{proof}

Next we translate the noncrossing conditions for $B$-diagonals into
conditions for the arrows representing them. 
\begin{proposition}
\label{proposition_across}
Suppose a pair of $B$-diagonals is represented by a pair of arrows as
defined in Definition~\ref{definition_arrow_representation}.
These $B$-diagonals cross
if and only if one of the following conditions is satisfied:
\begin{enumerate}[(1)]
\item Both arrows are backward and they cross.
\item Both arrows are forward and they do not nest.
\item One arrow is forward, the other one is backward, and
the backward arrow nests or crosses the forward arrow.
\item The head of one arrow is the tail of the other arrow.
\end{enumerate}    
\end{proposition}  
\begin{proof}
Suppose the arc representations of the two $B$-diagonals are
$\{[u_{1},v_{1}-1],[u_{1}+n+1,v_{1}+n]\}$ and $\{[u_{2},v_{2}-1],[u_{2}+n+1,v_{2}+n]\}$,
respectively. By Proposition~\ref{proposition_crossing_pi}, the represented $B$-diagonals
are crossing if and only if the images
$\pi([u_{1},v_{1}-1])$ and 
$\pi([u_{2},v_{2}-1])$
are not disjoint and do not contain each other.

We will consider three cases, depending on the
direction of the two arrows $(\pi(v_{1}-1),\pi(u_{1}))$ and
$(\pi(v_{2}-1),\pi(u_{2}))$. These arrows are either both forward, 
both backward or have opposite directions.

If both arrows are backward then neither the image
$\pi([u_{1},v_{1}-1])$ nor the image $\pi([u_{2},v_{2}-1])$
contain the point $n+1$.
Two such intervals intersect nontrivially in an interval of positive length
exactly when the corresponding arrows cross. They intersect in a
single point exactly when there is a vertex that is the tail of one of
the arrows and the head of the other arrow.

If both arrows are forward then both images
$\pi([u_{1},v_{1}-1])$ and $\pi([u_{2},v_{2}-1])$ contain
the point $n+1$, so they cannot be disjoint. They do not contain each
other exactly when the corresponding arrows are not nested. Note that a
pair of forward arrows such that the head of one arrow is the same as
the tail of the other is particular example of a pair of nonnested arrows.

Assume finally that one of the arrows, say $(\pi(v_{1}-1),\pi(u_{1}))$,
is a backward arrow and the other one, say $(\pi(v_{2}-1),\pi(u_{2}))$,
is a forward arrow. The image $\pi([u_{2},v_{2}-1])$
cannot be a
subset of $\pi([u_{1},v_{1}-1])$ as the first image contains
the point $n+1$ whereas the second does not.
The image $\pi([u_{1},v_{1}-1])$
is the interval $[\pi(u_{1}),\pi(v_{1}-1)]$,
whereas
the image $\pi([u_{2},v_{2}-1])$
is the union $(0,\pi(v_{2}-1)] \cup [\pi(u_{2}),n+1]$. 
The intersection of these two sets has either
zero, one or two connected components.
If one component is a single point
then this point is the head of one arrow
and the tail of the other.
If both components are non-trivial
intervals then the backward
arrow nests the forward arrow.
If they intersect in one interval
but the image $\pi([u_{1},v_{1}-1])$
does not contain
the image $\pi([u_{2},v_{2}-1])$
then the two arrows cross.
Finally, if the image $\pi([u_{1},v_{1}-1])$
does contain
the image $\pi([u_{2},v_{2}-1])$
then the arrows neither cross nor nest.
\end{proof}  
An immediate consequence of
Proposition~\ref{proposition_across}
and Lemma~\ref{lemma_facets}
is the following corollary.
\begin{corollary}
Noncrossing sets of $B$-diagonals correspond to subsets of vertices
contained in a facet of the Legendre polytope $P_{n}$. 
\end{corollary}

\section{The type $B$ associahedron represented
as a pulling triangulation}
\label{section_triangulation}

In this section we show that the arc representation given in
Definition~\ref{definition_arrow_representation}
constitutes
Simion's type~$B$ associahedron $\Gamma_{n}^{B}$
as a pulling triangulation of the boundary of the Legendre polytope $P_{n}$.
Our main result is the following.
\begin{theorem}
\label{theorem_r_pull}
Let $<$ be any linear order on the vertex set of $P_{n}$ subject to the
following conditions:
\begin{enumerate}
\item $(x_{1},y_{1})<(x_{2},y_{2})$
whenever $x_{1}-y_{1} > 0 > x_{2}-y_{2}$.
\item On the subset of vertices $(x,y)$ satisfying $x<y$, we have 
$(x_{1},y_{1})<(x_{2},y_{2})$  whenever the interval $[x_{1},y_{1}]=\{x_{1},
  x_{1}+1,\ldots, y_{1}\}$ is contained in the interval $[x_{2},y_{2}]$.
\item On the subset of vertices $(x,y)$ satisfying $x>y$, we have 
$(x_{1},y_{1})<(x_{2},y_{2})$  whenever the interval $[y_{1},x_{1}]$ is contained in
  the interval $[y_{2},x_{2}]$. 
\end{enumerate}
Then the arc representation of $\Gamma_{n}^{B}$ given in
Definition~\ref{definition_arrow_representation} is a pulling triangulation of the boundary
of the Legendre polytope $P_{n}$ with respect to $<$.
\end{theorem}
\begin{proof}
Fix any pulling order $<$ satisfying the above conditions.
The first condition requires all backward arrows to precede all
forward arrows.
The next two conditions require that for a pair
of nested arrows of the same direction the nested arrow should precede
the nesting arrow.
\newcommand{\lift}[1]{\raisebox{2.5mm}{#1}}
\begin{table}[t]
\begin{tabular}{|c||c|c|c|}  
\hline
  Order of nodes & Arrow first pulled & Edge & Not an edge\\
\hline
\hline  
\lift{$x_{1}<x_{2}<y_{1}<y_{2}$} & \lift{$(x_{2},y_{1})$} &
\begin{dependency}[theme=simple]
\begin{deptext} $x_{1}$\& $x_{2}$ \& $y_{1}$ \& $y_{2}$\\\end{deptext}
  \depedge[arc angle=30]{1}{4}{}
  \depedge[arc angle=30]{2}{3}{}
\end{dependency}
&
\begin{dependency}[theme=simple]
\begin{deptext} $x_{1}$\& $x_{2}$ \& $y_{1}$ \& $y_{2}$\\\end{deptext}
  \depedge[arc angle=30]{1}{3}{}
  \depedge[arc angle=30]{2}{4}{}
\end{dependency}
\\
\hline
\lift{$x_{1}<y_{1}<x_{2}<y_{2}$} & \lift{$(x_{2},y_{1})$} &
\begin{dependency}[theme=simple]
\begin{deptext} $x_{1}$\& $y_{1}$ \& $x_{2}$ \& $y_{2}$\\\end{deptext}
  \depedge[arc angle=30]{1}{4}{}
  \depedge[arc angle=30]{3}{2}{}
\end{dependency}
&
\begin{dependency}[theme=simple]
\begin{deptext} $x_{1}$\& $y_{1}$ \& $x_{2}$ \& $y_{2}$\\\end{deptext}
  \depedge[arc angle=30]{1}{2}{}
  \depedge[arc angle=30]{3}{4}{}
\end{dependency}
\\
\hline
\lift{$x_{1}<y_{1}<y_{2}<x_{2}$} & \lift{$(x_{2},y_{2})$} &
\begin{dependency}[theme=simple]
\begin{deptext} $x_{1}$\& $y_{1}$ \& $y_{2}$ \& $x_{2}$\\\end{deptext}
  \depedge[arc angle=30]{1}{2}{}
  \depedge[arc angle=30]{4}{3}{}
\end{dependency}
&
\begin{dependency}[theme=simple]
\begin{deptext} $x_{1}$\& $y_{1}$ \& $y_{2}$ \& $x_{2}$\\\end{deptext}
  \depedge[arc angle=30]{1}{3}{}
  \depedge[arc angle=30]{4}{2}{}
\end{dependency}
\\
\hline
\lift{$y_{1}<x_{1}<x_{2}<y_{2}$} & \lift{$(x_{1},y_{1})$} &
\begin{dependency}[theme=simple]
\begin{deptext} $y_{1}$\& $x_{1}$ \& $x_{2}$ \& $y_{2}$\\\end{deptext}
  \depedge[arc angle=30]{2}{1}{}
  \depedge[arc angle=30]{3}{4}{}
\end{dependency}
&
\begin{dependency}[theme=simple]
\begin{deptext} $y_{1}$\& $x_{1}$ \& $x_{2}$ \& $y_{2}$\\\end{deptext}
  \depedge[arc angle=30]{2}{4}{}
  \depedge[arc angle=30]{3}{1}{}
\end{dependency}
\\
\hline
\lift{$y_{1}<x_{1}<y_{2}<x_{2}$} & \lift{$(x_{1},y_{1})$ or $(x_{2},y_{2})$}&
\begin{dependency}[theme=simple]
\begin{deptext} $y_{1}$\& $x_{1}$ \& $y_{2}$ \& $x_{2}$\\\end{deptext}
  \depedge[arc angle=30]{2}{1}{}
  \depedge[arc angle=30]{4}{3}{}
\end{dependency}
&
\begin{dependency}[theme=simple]
\begin{deptext} $y_{1}$\& $x_{1}$ \& $y_{2}$ \& $x_{2}$\\\end{deptext}
  \depedge[arc angle=30]{2}{3}{}
  \depedge[arc angle=30]{4}{1}{}
\end{dependency}
\\
\hline
\lift{$y_{1}<y_{2}<x_{1}<x_{2}$} & \lift{$(x_{1},y_{2})$} &
\begin{dependency}[theme=simple]
\begin{deptext} $y_{1}$\& $y_{2}$ \& $x_{1}$ \& $x_{2}$\\\end{deptext}
  \depedge[arc angle=30]{3}{2}{}
  \depedge[arc angle=30]{4}{1}{}
\end{dependency}
&
\begin{dependency}[theme=simple]
\begin{deptext} $y_{1}$\& $y_{2}$ \& $x_{1}$ \& $x_{2}$\\\end{deptext}
  \depedge[arc angle=30]{3}{1}{}
  \depedge[arc angle=30]{4}{2}{}
\end{dependency}
\\
\hline
\end{tabular}
\vspace*{3mm}
\caption{Pairs of arrows that are edges or minimal nonfaces}
\label{table_six_way}  
\end{table}

Recall that~$\Gamma_{n}^{B}$ is a flag complex and its
minimal nonfaces are the pairs of crossing $B$-diagonals.
By Theorem~\ref{theorem_pull_flag} the pulling triangulation we
defined is also a flag complex. It suffices to show that the minimal
nonfaces are in bijection. Equivalently,  
for any pair of arrows $\{(x_{1},y_{1}),(x_{2},y_{2})\}$ that form an edge in
the pulling triangulation of $P_{n}$,
these arrows correspond to a pair of noncrossing
$B$-diagonals in~$\Gamma_{n}^{B}$.
By Proposition~\ref{proposition_across} this amounts to showing the
following: backward arrows cannot cross, forward arrows must nest, and
for a pair of arrows of opposite direction the backward arrow cannot
cross or nest the forward arrow.

Given a four-element subset $\{x_{1},x_{2},y_{1},y_{2}\}$
of $\{1, 2, \ldots, n+1\}$,
where $x_{1} < x_{2}$ and $y_{1} < y_{2}$,
consider the four arrows
in the set $\{x_{1},x_{2}\} \times \{y_{1},y_{2}\}$.
As seen in
Lemma~\ref{lemma_square}, these four arrows form the vertex set of a
square face of the Legendre polytope $P_{n}$, and only the diagonal
which contains the first vertex to be pulled is an edge of the pulling
triangulation.
Table~\ref{table_six_way}
lists all six possible orderings of this four element set.

In each of the six cases we note which vertex is pulled first,
which diagonal of the square becomes an edge in the 
triangulation and which diagonal does not become an edge.
Note that in the fifth row
of Table~\ref{table_six_way}
we have two possibilities for selecting the vertex to be pulled first.
However, these two vertices belong to the same diagonal.
In every case
we obtain that none of the pairs of arrows with distinct heads and tails
that is an edge corresponds to a pair of crossing
$B$-diagonals. 
\end{proof}

As a corollary we obtain Simion's polytopal result.
\begin{corollary}[Simion]
The Simion type $B$ associahedron $\Gamma_{n}^{B}$ is
the boundary complex of a simplicial polytope.
\end{corollary}
Since the associahedron of type $A$
is the link of a $B$-diagonal
of the form~$\{i,\overline{i}\}$, we obtain the following classical result;
see the work of Haiman, Lee~\cite{Lee} and Stasheff.
For a brief history, see the introduction
of~\cite{Ceballos_Santos_Ziegler}.
\begin{corollary}
The associahedron is
the boundary complex of a simplicial polytope.
\end{corollary}

We end this section by describing
the structure of all facets
of the Simion's type $B$ associahedron
in terms of arrows.
\begin{theorem}
\label{theorem_arrow}
A set of arrows $S=\{(x_{1},y_{1}),\ldots,(x_{n},y_{n})\}$ represents a facet of
Simion's type~$B$ associahedron $\Gamma_{n}^{B}$
if and only if the following conditions are satisfied: 
\begin{enumerate}
\item There is exactly one $k$ satisfying $1 \leq k \leq n+1$
such that $(k-1,k)$
(or $(n+1,1)$ if $k=1$) belongs the set~$S$. We call this $k$ the {\em
type} of the facet. 
\item Backward arrows do not nest any forward arrow, in particular, they
cannot nest $(k-1,k)$ if $k>1$. 
\item If $k=1$ then there is no forward arrow in the set $S$.
\item Forward arrows must nest. In particular, if $k>1$ then for each
each forward arrow $(x,y)\in S$ must satisfy $x\leq k-1$
and $y \geq k$. (Forward arrows must nest $(k-1,k)$.) 
\item
No head of an arrow in the set $S$
is also the tail of another arrow in $S$.
\item No two arrows cross.
\end{enumerate} 
\end{theorem}
\begin{proof}
Condition~(1) is equivalent to
Lemma~\ref{lemma_diameter}.
Except for Condition~(3), the remaining conditions are stated for all
faces in Proposition~\ref{proposition_across}. To prove Condition~(3), observe
that $k=1$ implies that the backward arrow $(n+1,1)$ belongs to
$S$. This arrow would nest any forward arrow, contradicting
Condition~(3) in Proposition~\ref{proposition_across}. 
\end{proof}

\section{Triangulating Cho's decomposition}
\label{section_Cho}

The type $A$ root polytope $P_{n}^{+}$
is the convex hull of the origin and the set of points
$\{e_{i} - e_{j} \: :\: 1 \leq i < j \leq n+1\}$.
Cho~\cite{Cho} gave a decomposition
of the Legendre polytope $P_{n}$
into $n+1$ copies of $P_{n}^{+}$ as follows. The symmetric group
$\SSSS_{n+1}$ acts on the Euclidean space $\Rrr^{n+1}$ by
permuting the coordinates, that is, the permutation
$\sigma \in \SSSS_{n+1}$ sends the
basis vector $e_{i}$ into $e_{\sigma(i)}$. Hence the permutation~$\sigma$
acts on the Legendre polytope $P_{n}$ by sending each $e_{i}-e_{j}$ into
$e_{\sigma(i)}-e_{\sigma(j)}$.
Cho's main result~\cite[Theorem~16]{Cho} is
the following decomposition.
\begin{theorem}[Cho]
The Legendre polytope $P_{n}$ has the decomposition 
$$ P_{n}  =  \bigcup_{k=0}^{n} \zeta^{k}(P_{n}^{+}) $$ 
where $\zeta$ is the cycle $(1,2,\ldots,n+1)$.
Furthermore, for $0 \leq k < r \leq n$
the polytopes $\zeta^{k}(P_{n}^{+})$ and $\zeta^{r}(P_{n}^{+})$
have disjoint interiors.
\label{theorem_Cho}
\end{theorem}
In this section we show that each copy $\zeta^{k}(P_{n}^{+})$ of
$P_{n}^{+}$ is the union of simplices of the triangulation given in
Definition~\ref{definition_arrow_representation}, representing the boundary complex
$\Gamma_{n}^{B}$ of Simion's type $B$ associahedron.  

\begin{theorem}
Every facet $F$ of the arc representation of $\Gamma_{n}^{B}$ given in
Definition~\ref{definition_arrow_representation} is contained in
$\zeta^{k-1}(P_{n}^{+})$ 
where $k$ is the unique arrow of the form $(k-1,k)$ in $F$
or $(n+1,1)$ if $k=1$.
Equivalently, the facet $F$ is contained in $\zeta^{k}(P_{n}^{+})$
exactly when it represents a facet of $\Gamma_{n}^{B}$
that contains the diagonal~$\{k,\overline{k}\}$.
\end{theorem}  
\begin{proof}
The polytope $P_{n}^{+}$ is the convex hull the origin and of all vertices that are
represented by backward arrows. The facets of type $1$ (as defined in
Theorem~\ref{theorem_arrow}) form a pulling triangulation of the part of the
boundary of $P_{n}^{+}$ that does not contain the origin. In fact, the
restriction of the pulling order to the backward arrows may be taken in the
{\em revlex order}, as defined in \cite[Definition~4.5]{Hetyei-Legendre}, giving
rise to the {\em standard triangulation} described
in~\cite{Gelfand-Graev-Postnikov}. A facet in our triangulation of $P_{n}$
belongs to the standard triangulation of the boundary of the part of
$P_{n}^{+}$ not containing the origin exactly when it has type $1$.

Observe next that the effect of $\zeta$ on the arrows (considered as
vertices of $P_{n}$) is adding $1$ modulo $n+1$ to the head and to the
tail of each arrow. Taking into account
Definition~\ref{definition_arrow_representation}
and
Remark~\ref{remark_arrows}, it is not difficult to see that the induced effect
on the arc representation is adding $1$ modulo $n+1$, that is, a
rotation. It is worth noting that to rotate each vertex of a regular
$(2n+2)$-gon into itself requires increasing each index by one $2n+2$
times. However, a centrally symmetric pair of arcs $\{[u, v-1], [u+n+1,
  v+n]\}$ is taken into itself already by $n+1$ such elementary
rotations. In the arc representation the facets containing only backward
arrows corresponding to facets containing the pair of arcs
$\{[1,n+1],[n+2,2n+2]\}$. The induced action of $\zeta^{k-1}$ takes this
pair into $\{[k,n+k],[n+k+1,2n+1+k]\}$. The conditions stated in
Theorem~\ref{theorem_across} are rotation-invariant, so the induced action
of $\zeta^{k-1}$ takes the family of all facets containing
$\{[1,n+1],[n+2,2n+2]\}$ into the family of all facets containing 
$\{[k,n+k],[n+k+1,2n+1+k]\}$.
\end{proof}

\section{A bijection between the faces of $\Gamma_{n}^{B}$ and Delannoy paths}
\label{section_bijection}

In Simion's original paper, the nontrivial computation of
the $f$-vector of the type $B$ associahedron~$\Gamma_{n}^{B}$ used the
recursive structure of the $B$-diagonals.
Simion asked if there is a direct way to obtain this
$f$-vector. Due to the unimodularity of the Legendre polytope
$P_{n}$, all pulling triangulations of the boundary give the same
$f$-vector.  A direct enumeration in the case of the lexicographic
pulling order, also known as the anti-standard triangulation,
is straightforward and well-known; see~\cite{Ardila,Hetyei-Legendre}.

In this section we establish a bijection between
the faces of $\Gamma_{n}^{B}$ and
balanced Delannoy paths of length $2n$ in such a way that
$(k-1)$-dimensional faces correspond to Delannoy paths with $k$
up steps, thus answering Simion's question.

\begin{definition}
A {\em balanced Delannoy path} of length $2n$ is a lattice path starting at $(0,0)$,
ending at $(2n,0)$
and using only steps of the following three types:
up steps $(1,1)$,
down steps $(1,-1)$
and horizontal steps $(2,0)$. 
A {\em Schr\"oder path} is a balanced Delannoy path that
never goes below the horizontal axis.   
\label{definition_Delannoy_Schroder}
\end{definition}
More generally, a Delannoy path is a path taking the steps
$(1,1)$, $(1,-1)$ and $(2,0)$ with no ending condition.
In this paper we will only work with balanced Delannoy paths.

Denote the up, down and horizontal steps
by $U$, $D$ and $H$, respectively.
We say that the {\em length} of the
letters $U$ and $D$ is $1$, whereas the length of the letter $H$ is $2$.
A Delannoy path is uniquely encoded by a word
in these three letters.
We call the associated word a Delannoy word.
A balanced Delannoy path of length $2n$
corresponds to a word
in which the number $k$ of the occurrences of $U$ is the
same as the number of occurrences of $D$,
and the number of occurrences of $H$ is $n-k$.
We call such a word {\em balanced}.
Hence the number of Delannoy paths of length $n$ with $k$ up steps is
given by the multinomial coefficient
$\binom{n+k}{k, k, n-k} = \binom{n+k}{k}\binom{n}{k}$
which is
the same as the number of $(k-1)$-dimensional faces in any pulling
triangulation of the boundary of~$P_{n}$;
see Lemma~\ref{lemma_any_pull}.

We represent the faces of $\Gamma_{n}^{B}$ as
a digraph on
the set of nodes $\{1, 2, \ldots,n+1\}$.
A contrapositive formulation of 
Proposition~\ref{proposition_across} is the following.
\begin{proposition}
A digraph represents a face of
Simion's type~$B$ associahedron $\Gamma_{n}^{B}$
exactly when the following conditions are satisfied:
\begin{enumerate}[(1)]
\item There are no crossings between arrows.
\item Forward arrows nest.
\item A backward arrow cannot nest a forward arrow.
\item No head of an arrow is the tail of another arrow.  
\end{enumerate}  
\label{proposition_contrapositive}
\end{proposition}  
See also Table~\ref{table_six_way} 
occurring in the proof of
Theorem~\ref{theorem_r_pull}.
Call such a set of arrows a {\em valid} digraph.
An example of a valid digraph is shown in Figure~\ref{figure_valid}.

\newcommand{\backwardarrow}[2]
{\node[anchor=east] at #1 (text) {};
\node[anchor=west] at #2 (description) {};
\draw[line width=0.3mm] (description) edge[out=135,in=45,->] (text);}

\newcommand{\forwardarrow}[2]
{\node[anchor=east] at #1 (text) {};
\node[anchor=west] at #2 (description) {};
\draw[line width=0.3mm] (description) edge[out=135,in=45,<-] (text);}

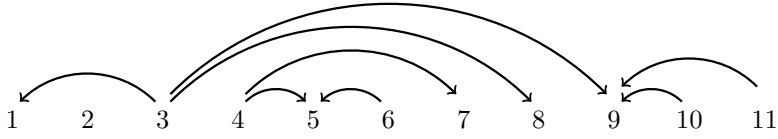
\begin{figure}[t]
\begin{tikzpicture}
\foreach \x in {1,...,11}
\node at (\x,0) {\small $\x$};

\backwardarrow{(1+0.1,0.1)}{(3-0.1,0.1)}
\backwardarrow{(5+0.1,0.1)}{(6-0.1,0.1)}
\backwardarrow{(9+0.1,0.1)}{(10-0.1,0.1)}
\backwardarrow{(9+0.1,0.3)}{(11-0.1,0.3)}

\forwardarrow{(3+0.1,0.1)}{(8-0.1,0.1)}
\forwardarrow{(3+0.1,0.2)}{(9-0.1,0.2)}
\forwardarrow{(4+0.1,0.1)}{(5-0.1,0.1)}
\forwardarrow{(4+0.1,0.2)}{(7-0.1,0.2)}

\end{tikzpicture}
\caption{A valid digraph on $11$ nodes.}
\label{figure_valid}
\end{figure}
For $W$ a subset of~$V$ let $A_{W}$ denote the
induced subgraph on $W$.
Let $\alpha \cdot \beta$ denote the concatenation of
the two words $\alpha$ and $\beta$.

Let $A$ be a valid digraph of backward arrows
on a non-empty finite set of nodes $V \subseteq \Ppp$
and let $v$ be the minimal element of $V$. 
We define the Schr\"oder word $\SP(A)$ by the following recursive definition.
\begin{itemize}
\item[(i)] If the set $V$ only consists of one node,
that is, $V = \{v\}$, let $\SP(A)$ be the empty word~$\epsilon$.
\item[(ii)]
If the minimal node $v$ is an isolated node of $A$, let 
$\SP(A)$ be the concatenation
$H \cdot \SP(A_{V-\{v\}})$.
\item[(iii)]
Lastly, when the node $v$ is not isolated then since $v$ is the
least element, there is necessarily a backward
arrow $(x,v)$ going into $v$. Let $w$ the smallest node
with an arrow to $v$, that is,
$w = \min(\{x \in V \: : \: (x,v) \in A\})$.
Define $\SP(A)$ by
$$    
\SP(A)
=
U \cdot
\SP(A_{V \cap (v,w]})
\cdot D \cdot
\SP(A_{V - (v,w]}) .
$$
\end{itemize}

\begin{lemma}
Let $A$ be a valid digraph on a set of nodes of size $n+1$
consisting of exactly $k$ backward arrows.
The Schr\"oder word $\SP(A)$ then has length $2 \cdot n$
and has exactly $k$ copies of the letters $U$ and $k$ copies of the letters $D$.
\label{lemma_easy_observation}
\end{lemma}
\begin{proof}
Both statements follow by induction on $n$.   
The first statement follows from the fact that
the two sets of nodes
$V \cap (v,w]$ and $V - (v,w]$
are complements of each other.
The second statement follows
by observing the noncrossing property ensures each arrow in $A$
is either the arrow $(w,v)$,
an arrow in $A_{V \cap (v,w]}$,
or an arrow in $A_{V - (v,w]}$.
\end{proof}

\begin{proposition}
The map $\SP$ is a bijection between
the set of all valid digraphs consisting only of backward arrows
on the set of $n+1$ nodes and all Schr\"oder words of length $2n$. 
\label{proposition_backwards}
\end{proposition}
\begin{proof}
Given a Schr\"oder word $\alpha$ of length $2n$
and a node set 
$V = \{v_{1} < v_{2} < \cdots < v_{n+1}\}$,
the inverse map is computed recursively as follows.
If $\alpha$ is the empty word~$\epsilon$ then
the inverse image is the isolated node~$v_{1}$.
If the word $\alpha$ begins with $H$, that is,
$\alpha = H \cdot \beta$ where $\beta$ has length $2n-2$,
compute the inverse image of $\beta$ on the nodes
$\{v_{2},v_{3}, \ldots, v_{n+1}\}$
and add the node $v_{1}$ as an isolated node.
Otherwise factor $\alpha$ uniquely as
$U \cdot \beta \cdot D \cdot \gamma$,
where $\beta$ and $\gamma$ are
Schr\"oder words of lengths $2p$ and $2n-2p-2$,
respectively.
Compute the inverse image of $\beta$
on the nodes $\{v_{2},v_{3}, \ldots, v_{p+2}\}$.
Similarly, compute the inverse image of $\gamma$ on the nodes
$\{v_{1}, v_{p+3}, v_{p+4}, \ldots, v_{n+1}\}$. Take the union of these
two digraphs and add the backward arrow $(v_{p+2},v_{1})$.
\end{proof}

We now extend the bijection $\SP$ to all valid digraphs.
In order to do this we introduce a
{\em twisting operation $\tw$}
on each digraph $A$ that has a forward arrow from the least element
$v \in V$ to the largest element $w \in V$. The {\em twisted digraph}
$\tw(A)$ is a digraph on the node set $V - \{v\}$ with arrow set
$$
\tw(A)
=
A_{V - \{v\}}
\cup
\{(w,z) \: : \: (v,z) \in A\}.
$$
In other words, we remove the least node $v$ and replace each forward
arrow $(v,z)$ starting at $v$ with a backward arrow $(w,z)$.
Note that
there is no backward arrow $(z,v)$ in $A$ as $v$ is the tail of 
the forward arrow $(v,w)$.
\begin{lemma}
Let $V$ be a node set of smallest node $v$ and largest node $w$.
The twisting map is a bijection from the set of valid digraphs on the set $V$
containing the forward arrow $(v,w)$ to the set of valid digraphs on
the set $V - \{v\}$.
\label{lemma_twisting}
\end{lemma}
\begin{proof}
We claim that the twisted digraph $\tw(A)$ is a valid digraph.
Begin to note that the restriction~$A_{V - \{v\}}$ is a valid digraph.
If there are no forward arrows in $A$ of the form $(v,z)$ where $z < w$,
the equality $\tw(A) = A_{V - \{v\}}$ holds and the claim is true.
Hence assume that there are forward arrows of the form $(v,z)$.
We need to show that adding the backward arrow $(w,z)$
the digraph is still valid.
We verify conditions (1) through (4) of
Proposition~\ref{proposition_contrapositive} in order.
If the arrow $(w,z)$ crosses an arrow~$(x,y)$
then the arrow $(v,z)$ already crossed this arrow in $A$,
verifying condition~(1).
Since we did not introduce any new forward arrows, 
condition~(2) holds vacuously.
Assume that the backward arrow $(w,z)$ nests a forward arrow $(x,y)$.
Then the two forward arrows $(v,z)$ and $(x,y)$ did not nest in $A$,
verifying~(3). The last condition~(4) holds directly, proving the claim.

The twisting map is one-to-one. Its inverse is given by
$$
\tw^{-1}(B)
=
\{(v,w)\}
\cup
B_{V \cap (v,w)}
\cup
\{(v,z) \: : \: (w,z) \in B\} .
$$
Note that in case the node $w$ was a `tail node'
the inverse map switches it back to being a `head node'.
\end{proof}

We now extend the bijection $\SP$ to a map $\DP$
which applies to all valid digraphs.
\begin{itemize}
\item[(i)]
If the valid digraph $A$ has no forward arrows,
let $\DP(A) = \SP(A)$.

\item[(ii)]
If the valid digraph $A$ has a forward arrow,
let $(x,y)$ be the forward arrow which nests the other forward arrows.
Let $v$ and $w$ be the minimal, respectively, the
maximal node of the node set~$V$,
that is,
$v = \min(V)$ and $w = \max(V)$.
Observe that $v \leq x < y \leq w$,
that is,
the two digraphs
$A_{V \cap [v,x]}$ and $A_{V \cap [y,w]}$
have no forward arrows.
Define $\DP(A)$ to be the concatenation
$$
\DP(A)
=
\SP(A_{V \cap [v,x]})
\cdot D \cdot  
\DP(\tw(A_{V \cap [x,y]}))
\cdot U \cdot  
\SP(A_{V \cap [y,w]}) .
$$
\end{itemize}
As an extension of Lemma~\ref{lemma_easy_observation}
we have the following lemma.
\begin{lemma}
Let $A$ be a valid digraph on a set $V$ of $n+1$ nodes
and assume that $A$ consists of $k$ arrows.
Then the balanced Delannoy word $\DP(A)$ has length $2 n$
and has exactly $k$ copies of the letters~$U$ and~$D$, respectively.
\end{lemma}
\begin{proof}
We only have to check the second case of the map $\DP$.
Assume that
$V \cap [v,x]$ and $V \cap [y,w]$
have cardinalities $a$, respectively $b$.
Then the middle part
$V \cap [x,y]$ has size $n-a-b+3$.
Thus 
$\tw(A_{V \cap [x,y]})$ has $n-a-b+2$ nodes.
Hence the total length is
$2 \cdot (a-1) + 2 \cdot (n-a-b+1) + 2 \cdot (b-1) + 2 = 2 n$.
For the second statement it is again enough to
check the second case.
Assume that the restricted digraphs
$A_{V \cap [v,x]}$ and $A_{V \cap [y,w]}$
have $c$, respectively $d$ arrows.
Since there is no arrow nesting the forward arrow~$(x,y)$,
the middle digraph has $k-c-d$ arrows.
Thus the twisted digraph has one less arrow,
namely $k-c-d-1$.
Hence the total number of $U$ letters
is $c + (k-c-d-1) + d + 1 = k$, proving the second claim.
\end{proof}

We now extend Proposition~\ref{proposition_backwards}
to all valid digraphs.
\begin{theorem}
The map $\DP$ is a bijection between
the set of all valid digraphs 
on a set of $n+1$ nodes and the set of all balanced Delannoy words of
length $2n$.
\label{theorem_forwards_and_backwards}
\end{theorem}
\begin{proof}
The inverse is computed as follows.
Let $\alpha$ be a balanced Delannoy word of length $2n$
and $V$ be a node set $\{v_{1} < v_{2} < \cdots < v_{n+1}\}$.
If the Delannoy word $\alpha$ is a Schr\"oder word, apply the inverse map of
Proposition~\ref{proposition_backwards}.
Otherwise, we can factor the Delannoy word $\alpha$ uniquely 
as $\beta \cdot D \cdot \gamma \cdot U \cdot \delta$,
where $\gamma$ is a balanced Delannoy word,
and $\beta$ and $\delta$ are Schr\"oder words.
Note that the factor $\beta \cdot D$ is uniquely
determined by the fact it is the shortest initial
segment in which the number of $D$ letters exceeds the number of $U$ letters,
in other words, the encoded path ends with the first down
step going below the horizontal axis.
By a symmetric argument, 
$U \cdot \delta$ is the shortest final segment with
one more $U$ than the number of $D$'s.
Assume that $\beta$, $\gamma$ and
$\delta$ have lengths $2p$, $2q$ and~$2n-2p-2q-2$, respectively.
Apply the inverse map from the proof of
Proposition~\ref{proposition_backwards}
to the words $\beta$ and $\delta$
to obtain
digraphs on $\{v_{1},v_{2}, \ldots, v_{p-1}\}$,
respectively
$\{v_{p+q+2}, v_{p+q+3}, \ldots, v_{n+1}\}$.

By recursion, apply the inverse of $\DP$ to $\gamma$ to obtain a valid
digraph on the node set $\{v_{p+2}, v_{p+3}, \ldots$, $v_{p+q+2}\}$.
Then apply the inverse of the twisting map to the digraph thus
obtained. Finally, take the union of these three digraphs.
\end{proof}

\begin{figure}[t]
\begin{tikzpicture}
\foreach \x in {1,...,11}
\node at (\x,0) {\small $\x$};

\backwardarrow{(1+0.1,0.1)}{(3-0.1,0.1)}
\backwardarrow{(5+0.1,0.1)}{(6-0.1,0.1)}
\backwardarrow{(5+0.1,0.3)}{(7-0.1,0.3)}
\backwardarrow{(8+0.1,0.1)}{(9-0.1,0.1)}
\backwardarrow{(9+0.1,0.1)}{(10-0.1,0.1)}
\backwardarrow{(9+0.1,0.3)}{(11-0.1,0.3)}

\end{tikzpicture}
\caption{The modified set of arrows obtained from the set in
  Figure~\ref{figure_valid}.}
\label{figure_twisted}
\end{figure}
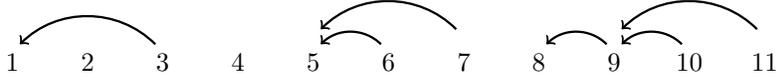

We now give non-recursive description of the bijection $\DP$.
First we encode a valid digraph with a multiset of
indexed letters.
Recall that a {\em weakly connected component} of a digraph
is a connected component in the graph obtained by
disregarding the direction of the arrows.
In our valid digraphs, the weakly connected components
are trees where each node is either a source or a sink.
\begin{definition}
\label{definition_multiset}  
Given a valid digraph $A$ on the node set $V$,
define the {\em associated multiset $M(A)$}
with elements from the set of indexed letters
$\{D_{x}, U_{x}, H_{x} \: : \: x \in \Ppp\}$
by the following three steps:
\begin{enumerate}
\item
For each $x < \max(V)$ which is a maximum element in
a weakly connected component 
we add the letter $H_{x}$ 
to the multiset $M(A)$.

\item
For each tail $x$ of a forward arrow,
consider the set
$\Head(x)=\{y>x \: : \: (x,y)\in A\}$, that is, the set of
heads of forward arrows with tail $x$.
Add a copy of the letter~$D_{x}$ to
$M(A)$,
also add a copy of the letter~$U_{w}$ for
$w=\max(\Head(x))$.
Remove the forward arrow~$(x,w)$.
For each remaining $y \in \Head(x)$ that is less than
$w$, replace the forward arrow $(x,y)$ with the backward arrow
$(w,y)$. The resulting set of arrows has only backward arrows.  

\item
For each head $y$ of a backward arrow,
consider the set
$\Tail(y) = \{x>y \: : \: (x,y) \in A\}$,
that is, the set of
tails of backward arrows with head $y$.
Add a copy of $U_{y}$ to $M(A)$.
Also add a copy of $D_{x}$ for $x \in \Tail(y)$
and add a copy of $U_{x}$ for
all but the maximum element of $\Tail(y)$.
\end{enumerate}
\end{definition}

\begin{example}
{\rm
For the valid set of arrows $A$ shown in
Figure~\ref{figure_valid},
in the first step we add the letters $H_{2}$ and $H_{7}$ to $M(A)$.
In the second step we add the letters
$D_{3}$, $U_{9}$, $D_{4}$, and $U_{7}$ to $M(A)$,
and we remove the forward arrows $(3,9)$ and $(4,7)$.
We also replace $(3,8)$ with $(9,8)$ and $(4,5)$ with $(7,5)$. 
We obtain the set of backward arrows shown in
Figure~\ref{figure_twisted}.
Note that this set of arrows does not need to be
valid anymore: in our example, $9$ is the head of $(10,9)$ and $(11,9)$
and it is also the tail of $(9,8)$. 
Finally, in step three we add the letters
$U_{1}$, $D_{3}$, $U_{5}$, $D_{6}$,
$U_{6}$, $D_{7}$, $U_{8}$, $D_{9}$,
$U_{9}$, $D_{10}$, $U_{10}$, and $D_{11}$ to $M(A)$. 
We end up with the multiset
$$
M(A)=
\{
U_{1},
H_{2},
D_{3},
D_{3}, 
D_{4},
U_{5}, 
D_{6},
U_{6}, 
D_{7}, 
U_{7}, 
H_{7},
U_{8}, 
D_{9},
U_{9},
U_{9}, 
D_{10}, 
U_{10}, 
D_{11}
\} .
$$
}
\end{example}
Define a linear order on the
indexed letters by the inequalities
$D_{x} < U_{x} < H_{x} < D_{x+1}$
for all positive integers $x$.
We obtain the lattice path as follows.
\begin{proposition}
The balanced Delannoy word $\DP(A)$ is
obtained from the multiset $M(A)$ by reading the
indexed letters in order and then omitting the subscripts.
\label{proposition_bijection_explicit}
\end{proposition}
For the set of arrows shown in Figure~\ref{figure_valid}, we
obtain the word $UHDDDUDUDUHUDUUDUD$. The lattice path encoded by this
word is shown in Figure~\ref{figure_lattice_path}. 
\begin{figure}[t]
\begin{tikzpicture}
\draw (-5,0) node (v1) {} -- (-4.5,0.5) -- (-3.5,0.5) -- (-3,0) -- (-2.5,-0.5) -- (-2,-1) -- (-1.5,-0.5) -- (-1,-1) -- (-0.5,-0.5) -- (0,-1) -- (0.5,-0.5) -- (1.5,-0.5) -- (2,0) -- (2.5,-0.5) -- (3,0) -- (3.5,0.5) -- (4,0) -- (4.5,0.5) -- (5,0) node (v2) {};
\draw [dashed] (v1) edge (v2);
\node at (-5,-0.2) {(0,0)};
\node at (-3,-0.5) {(4,0)};
\node at (3.2,-0.3) {(16,0)};
\node at (5,-0.2) {(20,0)};
\end{tikzpicture}
\caption{The lattice path associated to the digraph
in Figure~\ref{figure_valid}.}
\label{figure_lattice_path}
\end{figure}

A quick outline of a proof of
Proposition~\ref{proposition_bijection_explicit}
is as follows.
First prove the statement when there are no forward arrows.
In this case, each nonmaximal element $y$ of $\Tail(x)$
contributes two consecutive letters $D_{y} U_{y}$.
Next, observe that
a digraph and and its twisted digraph have
the same weakly connected components.
Furthermore, the arrows that moved under the
twisted operation still have the same set of heads.
Hence, after recording the horizontal steps in step~(1)
we may perform all twisting operations simultaneously
in step~(2).
We leave the remaining details to the reader.

\section{Concluding Remarks}

In recent paper, Cellini and Marietti~\cite{Cellini_Marietti_abelian} used
abelian ideals to produce a triangulation for various root polytopes.
In the case of type $A$, their
construction yields once again a lexicographic triangulation of
each face.  Restricting to the positive roots yields Gelfand, Graev
and Postnikov's anti-standard tree bases for the type $A$ positive
root polytope.  Is there an ideal corresponding to the reverse
lexicographic triangulation?

The $h$-vector of Simion's type~$B$ associahedron
may be computed from the $f$-vector using
elementary operations on binomial coefficients;
see~\cite[Corollary 1]{Simion}. 
\begin{lemma}[Simion]
The $h$-vector $(h_{0}, h_{1}, \ldots, h_{n})$ of
Simion's type~$B$ associahedron~$\Gamma_{n}^{B}$ satisfies 
$$
h_{i}=\binom{n}{i}^{2}\quad \mbox{for $0\leq i\leq n$}.
$$
\end{lemma}    
\noindent
One would like to find
a bijective proof of this result.
One plausible way of attack would be to find
an explicit shelling of $\Gamma^{B}_{n}$.

Finally, are there other interesting simplicial polytopes that
can be better understood as pulling triangulations of less
complicated polytopes?

\section*{Acknowledgments}

The first author was partially funded by
the National Security Agency grant H98230-13-1-028.
This work was partially supported by two grants from the
Simons Foundation
(\#245153 to G\'abor Hetyei and \#206001 to Margaret Readdy).
The authors thank the Princeton University Mathematics Department
where this research was initiated. 

\newcommand{\journal}[6]{{\sc #1,} #2, {\it #3} {\bf #4} (#5), #6.}
\newcommand{\book}[4]{{\sc #1,} ``#2,'' #3, #4.}
\newcommand{\bookf}[5]{{\sc #1,} ``#2,'' #3, #4, #5.}


\begin{thebibliography}{99}

\bibitem{Ardila}
\journal{F.\ Ardila, M.\ Beck, S.\ Ho\c{s}ten,
             J.\ Pfeifle, K.\ Seashore}
         {Root polytopes and growth series of root lattices}
         {SIAM J.\ Discrete Math.}
         {25}{2011}{360--378} 

\bibitem{Athanasiadis}
\journal{C.\ Athanasiadis}
         {$h^*$-vectors, Eulerian polynomials
          and stable polytopes of graphs}
         {Electron.\ J.\ Combin.}
         {11 (2)}{2004}{R6, 13 pp}

\bibitem{Bott_Taubes}
\journal{R.\ Bott and C.\ Taubes}
            {On the self-linking of knots. Topology and physics}
            {J.\ Math.\ Phys.}
            {35}{1994}{5247--5287}

\bibitem{Burgiel-Reiner}
\journal{H.\ Burgiel and V.\ Reiner}
          {Two signed associahedra}
          {New York J.\ Math.}
          {4}{1998}{83--95 (electronic)}
  
\bibitem{Ceballos_Santos_Ziegler}
\journal{C.\ Ceballos, F.\ Santos and G.\ M.\ Ziegler}
            {Many non-equivalent realizations of the associahedron}
            {Combinatorica}
            {35}{2015}{513--551} 
  
\bibitem{Cellini_Marietti_abelian}
\journal{P.\ Cellini and M.\ Marietti}
            {Root polytopes and Abelian ideals}
            {J.\ Algebr.\ Comb.}
            {2014}{39}{607--645}

\bibitem{Cho}
\journal{S.\ Cho}
           {Polytopes of roots of type $A_n$}
           {Bull.\ Austral.\ Math.\ Soc.}
           {59}{1999}{391--402} 

\bibitem{Clark-Ehrenborg}
\journal{E.\ Clark and R.\ Ehrenborg}
           {Excedances of affine permutations}
           {Adv.\ in Appl.\ Math.}
           {46}{2011}{175--191}

\bibitem{Cori-Hetyei}
\journal{R.\ Cori and G.\ Hetyei}
           {Counting genus one partitions and permutations}
           {S\'em.\ Lothar.\ Combin.}
           {70}{2013}{Art.\ B70e, 29 pp}

\bibitem{Gelfand-Graev-Postnikov}
            {\sc I.\ M.\ Gelfand, M.\ I.\ Graev, A.\ Postnikov,}
             Combinatorics of hypergeometric functions associated
             with positive roots,
             in Arnold-Gelfand Mathematical Seminars:
             Geometry and Singularity Theory,
             Birkh\"auser, Boston, 1996, 205--221.

\bibitem{Heller}
\journal{I.\ Heller}
            {On linear systems with integral valued solutions} 
            {Pacific J.\ Math.}
            {7}{1957}{1351--1364}

\bibitem{Heller-Hoffman}
\journal{I.\ Heller and A.\ J.\ Hoffman}
            {On unimodular matrices}
            {Pacific J.\ Math.}
            {4}{1962}{1321--1327}

\bibitem{Hetyei-Legendre}
\journal{G.\ Hetyei}
           {Delannoy orthants of Legendre polytopes}
           {Discrete Comput.\ Geom.}
           {42}{2009}{705--721} 

\bibitem{Hudson}
\bookf{J.\ F.\ P.\ Hudson}
          {Piecewise Linear Topology}
          {W.\ A.\ Benjamin, Inc.}
          {New York and Amsterdam}
          {1969}

\bibitem{Lee}
\journal{C.\ W.\ Lee}
            {The associahedron triangulations of the $n$-gon}
            {European J.\ Combin.}
            {10}{1989}{551--560}

\bibitem{Markl}
{\sc M.\ Markl,}
Simplex, associahedron, and cyclohedron.
Higher homotopy structures in topology and mathematical physics
(Poughkeepsie, NY, 1996),
235--265, Contemp.\ Math., 227,
Amer.\ Math.\ Soc., Providence, RI, 1999. 

\bibitem{Meszaros_I}
\journal{K.\ M\'esz\'aros}
            {Root polytopes, triangulations, and the subdivision algebra, I}
            {Trans.\ Am.\ Math.\ Soc.}
            {363}{2011}{4359--4382}

\bibitem{Meszaros_II}
\journal{K.\ M\'esz\'aros}
            {Root polytopes, triangulations, and the subdivision algebra, II}
            {Trans.\ Am.\ Math.\ Soc.}
            {363}{2011}{6111--6141}

\bibitem{Postnikov-P}
\journal{A.\ Postnikov}
            {Permutohedra, associahedra, and beyond}
            {Int.\ Math.\ Res.\ Not.\ IMRN}
            {}{2009}{1026--1106}

\bibitem{Simion}
\journal{R.\ Simion}
            {A type-B associahedron}
            {Adv.\ in Appl.\ Math.} 
            {30}{2003}{2--25} 

\bibitem{Stanley-Dec}
  \journal{R.\ P.\ Stanley}
             {Decompositions of rational convex polytopes}
             {Ann.\ Discrete Math.}
             {6}{1980}{333--342}

\end{thebibliography}
\end{document}